\documentclass[12pt]{amsart}%
\usepackage{amssymb}
\usepackage{amsmath}
\usepackage{geometry}
\usepackage{amsfonts}
\usepackage{graphicx}%
\newtheorem{theorem}{Theorem}
\theoremstyle{plain}

\newtheorem{lemma}[theorem]{Lemma}

\begin{document}
\title[On groups of infinite rank...]{On groups  in which subnormal  subgroups of infinite rank are commensurable
with some normal subgroup}
\author{Ulderico Dardano and Fausto De Mari}

\maketitle

\begin{abstract}
\noindent    We study soluble groups $G$  in which each subnormal
subgroup  $H$ with infinite rank is commensurable with a normal
subgroup, i.e. there exists a normal subgroup $N$ such that $H\cap
N$ has finite index in both $H$ and $N$. We show that if such a
$G$ is periodic, then all subnormal subgroups are commensurable
with a normal subgroup, provided either the Hirsch-Plotkin radical
of $G$ has infinite rank or $G$ is nilpotent-by-abelian (and has
infinite rank).
\end{abstract}


\section{Introduction and statement of results}

A group $G$ is said to be a $T$-group if normality in $G$ is a
transitive relation, i.e. if all subnormal subgroups are normal.
The structure of soluble $T$-groups was well described in the
1960s by Gasch\"{u}tz, Zacher and Robinson (see \cite{RT}). Then,
taking these results as a model, several authors have studied
soluble groups in which subnormal subgroups  have some embedding
property which \lq\lq approximates" normality. In particular,
Casolo \cite{C89} considered ${T}_*$-groups, that is groups in
which any subnormal subgroup $H$ has the property $nn$ (nearly
normal), i.e. the index $|H^G:H|$ is finite. Then Franciosi, de
Giovanni and Newell \cite{FdGN} considered ${T}^*$-groups, that is
groups in which any subnormal subgroup $H$ has the property $cf$
(core-finite, normal-by-finite), i.e. the index $|H:H_G|$ is
finite. Here, as usual, $H^G$ (resp. $H_G$) denotes the smallest
(resp. largest) normal subgroup of $G$ containing (resp. contained
in)  $H$

Recently in \cite{DDM},  in order to put those results in a common
framework, we considered $T[*]$-groups, that is groups in which
each subnormal subgroup $H$  is $cn$, i.e. commensurable with a
normal subgroup of $G$. Recall that two subgroups $H$ and $K$ are
called commensurable if $H\cap K$ has finite index in both $H$ and
$K$, hence both $nn$ and $cf$ imply $cn$. Clearly all the above
results rely on corresponding previous results on groups in which
{\sl all} subgroups are $nn$, $cf$, $cn$ resp. (see \cite{N,BLNSW,
CDR} resp.). A similar approach was adopted in \cite{DR5} where
finitely generated groups in which subnormal subgroups are inert
have been considered, where the term {\em inert} refers to a
different generalization of both $nn$, $cf$ (namely, an inert
subgroup is a subgroup which is commensurable with each of its
conjugates).

In the last decade, several authors have studied the influence on
a soluble group of the behavior of its subgroups of {\sl infinite
rank} (see for instance \cite{DDM1, DM} or the bibliography in
\cite{DFdGMT}). Recall that a group $G$ is said to have {\em
finite rank} $r$ if every finitely generated subgroup of $G$ can
be generated by at most $r$ elements, and $r$ is the least
positive integer with such property and {\em infinite rank} is
there is no such $r$.
 For example,  in \cite{DGMS} it was proved that
 {\em if $G$ is a periodic soluble group of infinite rank in which
 every subnormal subgroup of infinite rank is normal, then $G$ is a $T$-group indeed}.
Then in \cite{DFdGMT}, authors have considered groups of infinite
rank with properties $T_+$  ($T^+$, resp.), that is groups in
which the condition of being $nn$ (resp.  $cf$) is imposed only to
subnormal subgroups {\em with infinite rank}. In fact, it has been
shown that {\em a periodic soluble group of infinite rank $G$ with
property
$T_+$  ($T^+$, resp.) has the full $T_*$  ($T^*$, resp.) property, provided one of the following holds:\\
(A) the Hirsch-Plotkin radical of $G$ has infinite rank, \\ (B)
the commutator subgroup $G'$ is nilpotent.}

In this paper we show that a similar statement is true also for
the property $cn$. Moreover, by a corollary, we give  some further
information about the property $cf$ as well. Let us call
$T[+]$\textit{-group} a group in which each subnormal subgroup of
infinite rank is a $cn$-subgroup.

\bigskip
\noindent {\bf Theorem A}\ \ {\em Let $G$ be a periodic soluble
$T[+]$-group whose Hirsch-Plotkin radical has infinite rank. Then
$G$ is a $T[\ast]$-group. }

\bigskip
\noindent {\bf Corollary} \ \ {\em  Let $G$ be a periodic soluble
$T[+]$-group (resp. $T_+$-group) of infinite rank such that
$\pi(G')$ is finite. Then all subgroups of  $G$ are $cn$ (resp.
$cf$).}

\bigskip
\noindent {\bf Theorem B}\ \ {\em \label{B}Let $G$ be a periodic
$T[+]$-group of infinite rank with nilpotent commutator subgroup.
Then $G$ is a $T[\ast]$-group. }

\medskip
Note that if $G=A\rtimes B$ is the holomorph group of the additive
group $A$ of the rational numbers by the multiplicative group $B$
of positive rationals (acting by usual multiplication), then, as
noticed in \cite{DGMS}, the only subnormal non-normal subgroups of
$G$ are those contained in $A$ (which has rank $1$) so that $G$ is
$T[+]$. However all proper non-trivial subgroups of $A$ are not
$cn$, since if they were $cn$ then they were $cf$ (see
\cite{DDMR}, Proposition 1) contradicting the fact that $A$ is
minimal normal in $G$.

\medskip
Our notation and terminology is standard and can be found in
\cite{R72,R96}

\section{Proofs}

By a standard argument one checks easily that {\em if $H_1$ and
$H_2$ are $cn$- (resp. $cf$-) subgroups of $G$, then $H_1\cap H_2$
is likewise $cn$  (resp. $cf$). The same holds for  $H_1H_2$,
provided this set is a subgroup.}

\begin{lemma}
\label{1}Let $G$ be a $T[+]$-group and let $A$ be a subnormal
subgroup of $G$. If $A$ is the direct product of infinitely many
non-trivial cyclic subgroups, then any subgroup of $A$ is a
$cn$-subgroup of $G$.
\end{lemma}

\begin{proof}
Let $X$ be any subgroup of $A$, then $X$ is a subnormal subgroup
of $G$ and $X$ is likewise a direct product of cyclic groups (see
\cite{R96}, 4.3.16). In order to prove that $X$ is a $cn$-subgroup
of $G$ we may assume that $X$ has finite rank. Then there exist
subgroups $A_{1}$, $A_{2}$, $A_{3}$  of $A$ with infinite rank
such that $X\leq A_{3}$ and $A=A_{1}\times A_{2}\times A_{3}$.
Thus $XA_{1}$ and $XA_{2}$ are subnormal subgroups of infinite
rank, so that they are both $cn$-subgroups of $G$. Therefore
$X=XA_{1}\cap XA_{2}$ is likewise $cn$ in $G$.
\end{proof}

\begin{lemma}
\label{2}Let $G$ be a periodic $T[+]$-group. If $G$ contains an
abelian subnormal subgroup of infinite rank $A$, then $G$ is a
$T[\ast]$-group.
\end{lemma}

\begin{proof}
By hypothesis there exists a normal subgroup $N$ of $G$ which is
commensurable with $A$. Then $A\cap N$ has finite index in $AN$
and hence $N$ is an abelian-by-finite group of infinite rank. In
particular, $N$ contains a characteristic subgroup\ $N_{\ast}$ of
finite index which is an abelian group of infinite rank; hence
replacing $A$ by $N_{\ast}$ it can be supposed that $A$ is a
normal subgroup. Since $G$ is periodic and $A$ has infinite rank,
it follows that the socle $S$ of $A$ is a normal subgroup of $G$
which is the direct product of infinitely many non-trivial cyclic
subgroups. Application of Lemma \ref{1} yields that all subgroups
of $S$ are $cn$-subgroups of $G$ and hence by Lemma 2.8 of
\cite{CDR}, there exist $G$-invariant subgroups $S_{0}\le S_1$ of
$S$ such that $S_0$ and $S/S_1$ are finite and all subgroups of
$S$ lying between $S_{0}$ ad $S_1$ are normal in $G$

Let $X$ be any subnormal subgroup of finite rank of $G$. Then
$X\cap S_1$ is finite, hence $S_{2}=S_{0}(X\cap S_1)$ is likewise
finite. Since $S_{2}X$ is commensurable with $X$, we may assume
$S_2=\{1\}$. Clearly there exist subgroups $S_{3}$ and $S_{4}$
with infinite rank such that $S_1=S_{3}\times S_{4}$. Since both
$S_{3}$ and $S_{4}$ are normal subgroups of $G$, we have that both
$XS_{3}$ and $XS_{4}$ are subnormal subgroups of infinite rank of
$G$ and hence they are both $cn$. Thus $X=XS_3\cap XS_{4}$ is
likewise a $cn$-subgroup of $G$.
\end{proof}

Recall that any primary locally nilpotent group of finite rank is
a Chernikov group (see \cite{R72} Part 2, p.38).

\begin{lemma}
\label{3}Let $G$ a $T[+]$-group of infinite rank. If $G$ is a Baer
$p$-group, then $G$ is a nilpotent $T[*]$-group.
\end{lemma}
\begin{proof} Let $X$ be any subnormal subgroup of $G$ with
finite rank. Then $X$ is a Chernikov group  (see \cite{R72} Part
2, p.389) and $X$ contains an abelian divisible normal subgroup
$J$ of finite index. Hence $J$ is subnormal in $G$, and so $J^{G}$
is abelian and divisible (see \cite{R72} Part 1, Lemma 4.46). If
$A$ is any abelian subnormal subgroup of $G$, the subgroup
$J^{G}A$ is nilpotent and $[J,A]=\{1\}$ (see \cite{R72} Part 1,
Lemma~3.13). Since $G$ is generated by its subnormal abelian
subgroups, it follows that $J\leq Z(G)$ and so $X/X_{G}$ is
finite. This proves that $G$ is a $T[\ast ]$-group, hence
nilpotent (see \cite{DDM}, Proposition 20).
\end{proof}

The following lemma is probably well-known but we are not able to
find it in the literature, hence we write also the proof.

\begin{lemma}\label{zf} Let $G$ be a periodic finite-by-abelian
group of finite rank. Then $G/Z(G)$ is finite.
\end{lemma}

\begin{proof}
Clearly $C=C_G(G')$ is a normal subgroup of finite index of $G$
which is nilpotent and has finite rank; in particular, any primary
component of $C$ is a Chernikov group. % (see \cite{R72} Part 2, p.38).
Let $\pi=\pi(G')$ be the set of all primes $p$ such that $G'$
contains some element of order $p$. Then $\pi$ is finite and so
the subgroup $C_{\pi}$ is a Chernikov group; hence
$C_{\pi}Z(G)/Z(G)$ is finite (see \cite{R72} Part 1, Lemma 4.3.1).
On the other hand $C_{\pi'}$ is abelian, and so it follows that
$C/Z(C)$ is finite. Thus $G$ is both abelian-by-finite and
finite-by-abelian and hence $G/Z(G)$ is finite.
\end{proof}

\noindent {\em Proof of Theorem A.} Assume, for a contradiction,
that the statement is false and let $X$ be a subnormal subgroup
$G$ which is not a $cn$-subgroup; in particular, $X$ has finite
rank. Among all counterexamples choose $G$ in such a way that $X$
has the smallest possible derived length. Then the derived
subgroup $Y=X^{\prime}$ of $X$ is a $cn$-subgroup by the minimal
choice on the derived length of $X$; on the other hand, $Y$ has
finite rank and so $Y/Y_{G}$ is finite (see \cite{DDMR},
Proposition 1). Then $X/Y_{G}$ is a finite-by-abelian group of
finite rank, and hence its centre $Z/Y_{G}=Z(X/Y_{G})$ has finite
index in $X/Y_{G}$ by Lemma \ref{zf}. Thus $Z$ is a subnormal
subgroup of $G$ which has finite index in $X$, so that the index
$|Z:Z_{G}|$ is infinite and hence $Z$ cannot be a $cn$-subgroup of
$G$ (see \cite{DDMR}, Proposition 1). Since $Y_{G}$ has finite
rank, the Hirsch-Plotkin radical of $G/Y_{G}$ has infinite rank
and so $G/Y_{G}$ is also a counterexample; thus replacing $G$ with
$G/Y_{G}$ and $X$ with $Z/Y_{G}$ it can be supposed that $X$ is
abelian. Hence $X$ is contained in the Hirsh-Plotkin radical $H$
of $G$.

Let $P$ any primary component of $H$, and suppose that $P$ has
infinite rank. If $F$ is the Fitting subgroup of $P$, then $F$ is
nilpotent by Lemma \ref{3}. Let $A$ be a maximal abelian normal
subgroup of $F$, then $A=C_{F}(A)$ (see \cite{R72} Part 1, Lemma
2.19.1) and so $A$ has infinite rank (see \cite{R72}\ Part 1,
Theorem~3.29). Hence $G$ is a $T[\ast]$-group by Lemma \ref{2}.
This contradiction proves that each primary component of $H$ has
finite rank. In particular, as $H$ has infinite rank, there exist
$H_{1}$ and $H_{2}$ subgroups of infinite rank such that
$H=H_{1}\times H_{2}$ and\ $\pi(H_{1})\cap\pi(H_{2})=\emptyset$.
By the same reason, for $i\in\{1,2\}$, two subgroups of infinite
rank $H_{i,1}$ and $H_{i,2}$ can be found such that
$H_{i}=H_{i,1}\times H_{i,2}$ and
$\pi(H_{i,1})\cap\pi(H_{i,2})=~\emptyset$. If $i,j\in\{1,2\}$ and
$i\neq j$, considered $\pi_{i}=\pi(H_i)$ and
denoted by $X_{i}$ the $\pi_{i}$-component of $X$, the subgroups $X_{i}%
H_{j,1}$ and $X_{i}H_{j,2}$ are subnormal subgroups of infinite
rank of $G$, so that they are both $cn$-subgroups of $G$ and hence
$X_{i}=X_{i}H_{j,1}\cap X_{i}H_{j,2}$ is likewise a $cn$-subgroup
of $G$. Therefore $X=X_{1}X_{2}$ is a $cn$-subgroup of $G$ and
this final contradiction concludes the proof. \qed
\\

\noindent {\em Proof of Corollary.} One may refine the derived
series of $G'$ to a series $G_1=G'\ge\cdots\ge G_n=\{1\}$ whose
factors $A_i=G_{i}/G_{i+1}$ are $p$-groups (for possibly different
primes). Let  $C_i=C_G(A_i)$ for each~$i$. If $A_i$ has finite
rank, then $A_i$ is a Chernikov group and the same holds for
$G/C_i$ as a periodic group of automorphisms of a Chernikov group
(see \cite{R72} Part 1, Theorem 3.29).
 If $A_i$ has infinite rank, then Lemma \ref{2} yields that each subgroup of  $A_i$ is a $cn$-subgroup (resp $cf$-) of $G$.
 Hence, according to Proposition 14 in \cite{DDM}, $G/C_i$ is finite as a periodic group of power automorphisms of $p$-groups
  (see \cite{RT}, Lemma 4.1.2). Thus if $C$ is the intersection of all $C_i$'s, then $G/C$ is a Chernikov group and therefore has finite rank.
  It follows that $C$ has infinite rank. Om the other hand, $C$ is nilpotent by a well-known fact (see \cite{H}).
  Then by Theorem $A$, the group $G$ has property $T[*]$ (resp. $T_*$). Further, by Theorem~15 in \cite{DDM}, $G$ all subgroups are $cn$ (resp. $cf$).
\qed

\medskip
\noindent {\em Proof of Theorem B.} Assume that the statement is
false. As in the first part of proof of Theorem~A, there exists a
counterexample $G$ containing an abelian subnormal subgroup $X$
that is not a $cn$-subgroup; in particular, $X$ has finite rank
and the index $|X:X_{G}|$ is infinite. Then $L=XG^{\prime}$ is a
nilpotent normal subgroup and hence has finite rank by Theorem A.
Let $p\in\pi (X)$, then $L/L_{p^{\prime}}$ is a nilpotent
$p$-group of finite rank and hence it is a Chernikov group; thus
$G/C_{G}(L/L_{p^{\prime}})$ is finite (see
\cite{R72} Part 1, Corollary p.85) and hence $C_{G/L_{p^{\prime}}%
}(L/L_{p^{\prime}})$ is a nilpotent normal subgroup of infinite
rank of $G/L_{p^{\prime}}$. Thus Theorem A yields that
$G/L_{p^{\prime}}$ is a $T[\ast]$-group. Therefore
$X_{p}L_{p^{\prime}}$ is a $cn$-subgroup of $G$, and hence it is
even $cf$ because it has finite rank (see \cite{DDMR}, Proposition
1). The $p$-component of the core $(X_{p}L_{p^{\prime}})_{G}$ of
$X_{p}L_{p^{\prime}}$in $G$ is $G$-invariant, it coincides with
the subgroup $X_{p}\cap(X_{p}L_{p^{\prime}})_{G}$ and so has
finite index in $X_{p}$, therefore $X_{p}$ is $cf$. In particular,
the set $\pi$ of all primes $p$ in $\pi(X)$ such that $X_{p}$ is
not normal in $G$ is infinite. Replacing $G$ by
$G/L_{\pi^{\prime}}$ it can be supposed that $\pi=\pi(L)$. Then
there exists an infinite subset $\pi_{0}$ of $\pi$ such that
$G/L_{\pi_{0}^{\prime}}$ contains a nilpotent normal subgroup of
infinite rank (see \cite{DFdGMT}, Corollary 11); hence
$G/L_{\pi_{0}^{\prime}}$ is a $T[\ast]$-group by Theorem A.
Therefore $X_{\pi_{0}}L_{\pi_{0}^{\prime}}$ is a $cn$-subgroup of
$G$ and so even a $cf$-subgroup (see \cite{DDMR}, Proposition 1);
hence $X_{\pi_{0}}$ is $cf$ and this is a contradiction because
$X_{p}$ is not normal in $G$ for each $p\in\pi_{0}$. \qed

\bigskip

\bigskip\noindent
{\small
{Ulderico Dardano, Fausto De Mari\\
%{\rm (corresponding author)} }\\
{Dipartimento di Matematica e
Applicazioni ``R.Caccioppoli'', }\\
{Universit\`a di Napoli ``Federico
II'', \\ Via Cintia - Monte S. Angelo, I-80126 Napoli, Italy}\\
{ email: dardano@unina.it , fausto.demari@unina.it }
}

\end{document}